\documentclass[11pt]{amsart}
\usepackage{bm}
\usepackage{amssymb}
\usepackage{graphicx}	
\usepackage{xcolor}
\usepackage{amsmath}
\usepackage{enumerate}
\usepackage{blindtext}
\usepackage{multirow}
\usepackage{mathrsfs}

\newtheorem*{notn}{Notations}
\newtheorem{thm}{Theorem}[section]
\newtheorem{Def}[thm]{Definition}
\newtheorem{Rem}[thm]{Remark}
\newtheorem{Cor}[thm]{Corollary}

\newtheorem{Pro}[thm]{Proposition}

\newtheorem{ex}[thm]{Example}
\newtheorem{prob}[thm]{Problem}
\newtheorem{lemma}[thm]{Lemma}

\newcommand{\bdfn}{\begin{Def} \rm}
\newcommand{\edfn}{\end{Def}}

\newcommand{\tfae}{the following are equivalent}

\newcommand{\ra}{\rightarrow}
\newcommand{\Ra}{\Rightarrow}

\newcommand{\Lglra}{\Longleftrightarrow}
\newcommand{\es}{\emptyset}
\newcommand{\ci}{\subseteq}

\newcommand{\al}{\alpha}
\newcommand{\be}{\beta}
\newcommand{\de}{\delta}
\newcommand{\e}{\varepsilon}

\newcommand{\la}{\lambda}

\newcommand{\Si}{\Sigma}
\newcommand{\ga}{\gamma}
\newcommand{\Ga}{\Gamma}
\newcommand{\La}{\Lambda}
\newcommand{\mb}{\mathbb}

\newcommand{\sm}{\setminus}

\newcommand{\iy}{\infty}

\newcommand{\beqa}{\begin{eqnarray*}}
\newcommand{\eeqa}{\end{eqnarray*}}
\newcommand{\vertiii}[1]{{\left\vert\kern-0.25ex\left\vert\kern-0.25ex\left\vert #1 
    \right\vert\kern-0.25ex\right\vert\kern-0.25ex\right\vert}}
\newcounter{cnt1}
\newcounter{cnt2}
\newcounter{cnt3}
\newcounter{cnt4}
\newcommand{\blr}{\begin{list}{$($\roman{cnt1}$)$} {\usecounter{cnt1}
\setlength{\topsep}{0pt} \setlength{\itemsep}{0pt}}}
\newcommand{\blR}{\begin{list}{\Roman{cnt4}.\ } {\usecounter{cnt4}
\setlength{\topsep}{0pt} \setlength{\itemsep}{0pt}}}
\newcommand{\bla}{\begin{list}{$(\alph{cnt2})$} {\usecounter{cnt2}
\setlength{\topsep}{0pt} \setlength{\itemsep}{0pt}}}
\newcommand{\bln}{\begin{list}{$($\arabic{cnt3}$)$} {\usecounter{cnt3}
\setlength{\topsep}{0pt} \setlength{\itemsep}{0pt}}}
\newcommand{\el}{\end{list}}

\begin{document}
\title[On Hahn-Banach smoothness of $L_1$-preduals]{On Hahn-Banach smoothness of $L_1$-preduals and related $w^*-w$ points of continuity of unit balls of dual spaces}
\author{Sainik Karak
}

\newcommand{\Addresses}{{
  \bigskip
  \footnotesize

  Sainik Karak, \textsc{Indian Institute of Technology Hyderabad,
    India}\par\nopagebreak
  \textit{E-mail address}, Sainik Karak: \texttt{ma22resch11001@iith.ac.in}

  \medskip
}}

\subjclass[2000]{Primary 46A22, 46B20, 46E15 Secondary 46B22, 46M05 \hfill \textbf{\today} }
\keywords{Hahn-Banach smooth,
	weak Hahn-Banach smooth, $L_1$-predual, $M$-embedded spaces, point of continuity}

\begin{abstract}
This article aims to examine the Hahn-Banach smoothness of Banach spaces and its connections to various geometrical aspects. We examine the circumstances that allow linear functionals to have unique norm-preserving extensions, with particular attention to the behavior of these properties in $L_1$-preduals and in spaces of affine continuous functions. Banach spaces which are $L_1$-preduals and also Hahn-Banach smooth are completely characterized. It is demonstrated that if $X$ is an $M$-embedded space then $X^*$ admits a predual which is not weakly Hahn-Banach smooth. It is derived that, when $S$ is a compact convex set where each point in $ext(S)$ is a limit point of $ext(S)$ and also represents a split face, no subspace of $A(S)$ retains the property-$(wU)$ in $A(S)^{**}$. Furthermore, when $X=C_0(L)$, in the context of a locally compact Hausdorff space $L$, the continuity of the identity mapping $I:(B_{X^*},w^*)\to (B_{X^*},w)$ in $ext (B_{X^*})$ significantly influences the subspaces of $X$ that have unique extension property in $X^{**}$. Collectively, this study provides structural characterizations of specialized geometric property, so called Hahn-Banach smoothness, and offers solutions to some natural problems enlisted at the beginning that involve spaces that are $L_1$-preduals and also spaces that are $M$-embedded.
\end{abstract}

\maketitle

\section{Introduction}
\subsection{Preliminaries:}
The problem of extending linear functionals, with the Hahn-Banach theorem at the forefront, is one of the classics in Banach space theory.
A starting point of the literature in this area is a seminal paper by Robert Phelps (\cite{Phelps}) in which the author investigated the so-called {\it property-$(U)$} for the first time. Numerous authors have investigated this property and its several variations. In this paper we focus on investigating the following properties.

\bdfn
Let $Y$ be a subspace of a Banach space $X$. $Y$ is said to have,
\bla
\item \textit{property-$(U)$} if every linear functional on $Y$ has a unique norm preserving extension over $X$.
\item \textit{property-$(wU)$} if every linear functional on $Y$ which attains its norm has a unique norm preserving extension over $X$.
\el
\edfn

When a Banach space $X$ has property-$(U)$  ($(wU)$) in its canonical embedding into $X^{**}$, $X$ is said to be {\it Hahn-Banach smooth} ({\it weakly Hahn-Banach smooth}). It is well-known that both properties are hereditary, one can check related discussion in a recent article \cite{PR}.
It is not known whether a Banach space is Hahn-Banach smooth if it is known to be weakly Hahn-Banach smooth. For a discrete set $\Ga$, $c_0(\Ga)$ is an example of Hahn-Banach smooth space. Here $c_0(\Ga)$ consists of all $\varphi:\Ga\ra\mb{R}$, so that the set $\{\ga:|\varphi (\ga)|>\e\}$ is at most finite, for $\e>0$. $K(\ell_p)$, $1<p<\iy$, the space of all compact operators on $\ell_p$, are also examples of Hahn-Banach smooth spaces. On the other hand $c$, the space of all convergent scalar sequences, $A(S)$, the scalar valued affine continuous functions defined on an infinite dimensional compact convex set $S$ are some examples which are not (weak) Hahn-Banach smooth spaces.

For a Banach space $X$, by $NA_1(X)$ we denote the set of functionals in $X^\ast$ of unit norm that attain their norm on the unit ball of $X$. When $Y$ is a subspace of $X$ and $y^*\in Y^*$, by $HB(y^*)$ we denote the set $\{x^*\in X^*:x^*|_{Y}=y^*, \|y^*\|=\|x^*\|\}$. $S_X, B_X$ stand for the closed unit sphere and closed unit ball of $X$ respectively. By $ext (B_X)$ we mean the set of all extreme points of $B_X$. For two Banach spaces $X_1$ and $X_2$, by $X_1\cong X_2$ we mean $X_1$ is isometrically isomorphic to $X_2$. The so called weak and weak-$*$ topologies are denoted by $w$ and $w^*$ respectively.
We denote the unit vector $(0,\ldots,0,1,0,\ldots)$, where $1$ appears in the $n$-th entry of the classical sequence spaces, viz. $c_0, c, \ell_1$, by $e_n$.

In \cite{FS}, Sullivan studied Hahn-Banach smoothness of Banach spaces. It is well-known that spaces which are Hahn-Banach smooth share nice geometric properties; viz. they are {\it Asplund spaces}. We refer the reader to Chapter~7 of \cite{FHHM} for this notion and the vast literature related to such spaces.

We refer to the following characterizations for Hahn-Banach smooth spaces from \cite{HL}. We also refer to the article \cite{JH} where the notion of the {\it Asymptotically norming property} (in short, ANP) was introduced.

\bdfn\label{D4}
\bla
\item A sequence $(f_n)$ in $S_{X^*}$ is said to be $w^*$-asymptotically normed if for any $\e>0$ there exists $x\in B_X$ such that $f_n(x)>1-\e$ for all but finitely many $n$'s.
\item The above $(f_n)$ is said to be $w^*$-ANP-III if $\cap_n \overline{conv}\{f_k:k\geq n\}\neq\es$.
\item $X^*$ is said to be $w^*$-ANP-III if any asymptotically normed sequence in $S_{X^*}$ is $w^*$-ANP-III.
\el
\edfn

In \cite{HL}, Hu and Lin considered the notion $w^*$-ANP-III if there exists an equivalent norm in $X^*$ where the above condition holds.

\begin{thm}\cite[Theorem~3.2]{HL}\label{T5}
Let $X$ be a Banach space. Then \tfae.
\bla
\item $X$ is Hahn-Banach smooth.
\item The topologies $w$ and $w^*$ coincide on $S_{X^*}$.
\item $X^*$ has $w^*$-ANP-III.
\el
\end{thm}

This article examines the topologies of $w$ and $w^*$ on the surface of the dual unit ball. The collection of points at which the identity mapping $I:(B_{X^*},w^*)\ra (B_{X^*},w)$ is continuous corresponds to the locations where the two topologies align. When there is no chance of misunderstanding, we call this set the set of {\it points of continuity} (in short PC). We cite a recent work by Daptari et al. \cite{DMR} in this regard. When $X^*=\ell_1$, several consequences arise from the $w^*$-topologies induced by $c_0$ and $c$. It is important to note that the canonical image of $c$ in its bidual $\ell_\infty$ can be identified with $\hat{c}=\{(\al_n)_{n=0}^\infty : \lim_{n} \al_n \text{ exists and } \al_0 = \lim_{n} \al_n\}$ (see \cite[Exercise~II.2.31]{FHHM}). The set of PC for the $w^*$-topology on $\ell_1$, when induced by $c$, is provided by $\{(a_n)\in S_{\ell_1}: (a_n) \mbox{~is finitely supported}\}$, according to the authors of \cite[Remark~9.5]{DMR}. This is a false claim. In fact, $e_1\in S_{\ell_1}$, which maps each sequence $(\al_n)\in \hat{c}$ to its limit, specifically $\lim_n\al_n$, possesses a minimum of two distinct Hahn-Banach extensions: $e_1$ and a Banach limit $L$ defined over $\ell_\infty$. It is evident that both assign $(\al_n)$ to its limit. The properties of the Banach limits can be accessed from p.82 of \cite{JBC} and also from p.81 of \cite{FHHM}.
This discussion is revisited in Remark~\ref{R1}. Theorem~\ref{T3} identifies a set $\Theta\ci S_{\ell_1}$ where the identity map $I:(B_{\ell_1},w^*)\ra (B_{\ell_1},w)$ is continuous. Here $(B_{\ell_1},w^*)$ represents the space $B_{\ell_1}$ with a $w^*$-topology induced by $c$.

\bdfn\label{D3}
A Banach space $X$ is said to be an {\it $L_1$-predual} if $X^*\cong L_1(\mu)$ for some measure space $(\Omega,\Sigma,\mu)$.
\edfn

A major part of this study is devoted to discussing the identification of spaces that are $L_1$-preduals and also Hahn-Banach smooth.
A large class of Banach spaces, including spaces of type $c_0(\Ga)$, where $\Ga$ is a discrete space, falls under this category. We refer to Chapters~ 6 and 7 of \cite{HEL}, the monograph by Lacey, for characterizations
of these spaces and their properties. The following notion is used frequently in this study.

\bdfn\cite{HWW}\label{D1}
\bla
\item A subspace $Y$ of a Banach space $X$ is said to be an {\it $M$-ideal of $X$} if there exists a subspace $Z$ such that $X^*=Y^\perp\oplus_{\ell_1}Z$.
\item A Banach space $X$ is said to be {$M$-embedded} if the canonical image of $X$ is an $M$-ideal in $X^{**}$.
\el
\edfn

See Chapter~ 3 in \cite{HWW} for various examples and properties of spaces that are $M$-embedded. The spaces stated above $c_0(\Ga), K(\ell_p)$ are some natural examples of $M$-embedded spaces. 
It clearly follows that $M$-embedded spaces are Hahn-Banach smooth. Similar to the Hahn-Banach smoothness, this property is also hereditary.
We refer to the monographs  \cite{HWW} and \cite{FHHM} for any unexplained terminology and notation used in this article. 

\subsection{Foundations and summary of this work:}
The following problems may serve as motivation for undertaking this research.

\begin{prob}\label{P10}
Can one identify a minimal subset of $B_{X^{*}}$, uniqueness of extensions of these functionals over $X^{**}$ ensures uniqueness of extensions of all functionals in $X^{*}$?
\end{prob}

\begin{prob}\label{P7}
Let $X$ be an $L_1$-predual. Under what assumption is $X$ Hahn-Banach smooth?
\end{prob}

\begin{prob}\label{P8}
Let $I:(B_{\ell_1},w^*)\ra (B_{\ell_1},w)$ be the identity mapping, where the $w^*$-topology on $\ell_1$ is induced by $c$. What are the points in $S_{\ell_1}$ where the map $I$ is continuous?
\end{prob}

\begin{prob}\label{P9}
To what extent can the geometric structures of $X$ be identified when the identity mapping $I:(B_{X^*},w^*)\ra (B_{X^*},w)$ is continuous on $ext (B_{X^*})$? Furthermore, if no points in $ext (B_{X^*})$ serve as PC for $I$, what geometric structures of $X$ (or its subspaces) are realized in relation to the extensions of functionals?
\end{prob}

The paper is organized as follows.  In section 2, we examine fundamental characteristics of extensions from $X$ to $X^{**}$ concerning linear functionals $x^* \in X^*$. A key aspect of the article \cite{Phelps} is that for $x^*\in X^*$ and a subspace $Y$ of $X$, the set $HB(x^*|_Y)$ can be expressed as $\{x^*-y^\perp:y^\perp\in P_{Y^\perp}(x^*)\}$. Here $P_{Y^\perp}(x^*)=\{y^\perp\in Y^\perp:\|x^*-y^\perp\|=d(x^*,Y^\perp)\}$. Several consequences can be extracted in this context when the canonical image of $X$ is considered as a subspace of $X^{**}$. We answer Problem~\ref{P10} when $X$ is a dual space. An uncountable family of non-isometric preduals of $\ell_1$ are identified which are Hahn-Banach smooth.

In section~3 we completely characterize $L_1$-preduals which are Hahn-Banach smooth. It is found if $(\Omega,\Si,\mu)$ is a measure space and there exists a predual of $L_1(\mu)$ which is Hahn-Banach smooth then $\mu$ must be of specific characteristics.

We answer Problem~\ref{P8} completely in section~4. In section~5 it is found that if $X$ is an $M$-embedded space then there always exists a predual of $X^*$ which is not weakly Hahn-Banach smooth.

Finally, various instances are discussed in the context of PC in section~5. We extend our observations to the spaces of affine functions on compact convex sets. The answer to Problem~\ref{P10} leads to Problem~\ref{P9}.
We answer Problem~\ref{P9} and discuss its various facets in this section.

\section{$U$-embedding of $X$ in $X^{**}$}

This section showcases that the property-$(U)$ (or $(wU)$) of the canonical image of a dual space in its bidual results in the trivial case. We first demonstrate that Theorem~\ref{T5} requires $X$ to be reflexive if it is a dual space. 

\begin{Pro}\label{P6}
Let $X$ be a Banach space such that $X$ has weak-ANP-III. Then $X$ is reflexive. Hence $X$ is reflexive if $X^{**}$ has $w^*$-ANP III.
\end{Pro}
\begin{proof}
Let $f\in S_{X^*}$. Suppose that $(f_n)\ci NA_1(X)$ is such that $f_n\ra f$ in $X^*$ and there exists $(x_n)\ci S_X$ such that $f_n(x_n)=\|f_n\|$. Then $(x_n)\ci S_X$ is asymptotically normed by $\{f\}$. As $X$ has weak-ANP-III, from \cite[Lemma~ 2.3]{HL}, the sequence $(x_n)$ has a weak convergence subsequence. Let us assume  $(x_{n_k})$ be a weak convergent subsequence which converges weakly to $x_0$. Hence $f(x_0)=\lim_k f(x_{n_k})$. Now
$|f(x_{n_k})|=|(f-f_{n_k})(x_{n_k})+f_{n_k}(x_{n_k}))|>\|f_{n_k}\|-\e$, for all $k$ sufficiently large. This leads to $f(x_0)=\|f\|$ and hence $\|x_0\|=1$. This concludes $f\in NA_1(X)$. Since $f$ is arbitrary, $X$ is reflexive.
\end{proof}

We now encounter a few cases where despite the failure of the Hahn-Banach smoothness of $X$, it is possible to find a copy of $X$ in $X^{**}$ other than the canonical one, which satisfies the unique Hahn-Banach extension property. The following example~\ref{E1}, makes it precise. Recall the definition of the James boundary in Banach spaces. We refer the reader to p.132 of \cite{FHHM} in this context.

\bdfn
Let $B\ci B_{X^*}$, where $X$ is a Banach space. $B$ is said to be a James boundary of $X$ if for any $x\in X$ there exists $x^*\in B$ such that $|x^*(x)|=\|x\|$.
\edfn

Let $X$ be a non-reflexive Banach space. Clearly, there exists $x^*\in S_{X^*}$, $x^*\notin NA_1(X)$. Let $B$ be a James boundary for $X^*$.
Finally, there exists $x^{**}\in B$ such that $|x^{**}(x^*)|=1$.    
Let $J_0:X\hookrightarrow X^{**}$ and $J_2:X^{**}\hookrightarrow X^{(4)}$ be the canonical embedding maps. Then it is easy to observe that $J_2(x^{**})|_{X^*}=J_0^{**}(x^{**})|_{X^*}$, although $J_2(x^{**})\neq J_0^{**}(x^{**})$. In fact there exists $x^\perp\in X^{***}$, where $x^\perp|_X=0$ and  $x^\perp (x^{**})\neq 0$ which in other words is $J_2(x^{**})(x^\perp)\neq 0$. On the other hand, we have $J_0^{**}(x^{**})(x^\perp)=0$. Consequently, in this case $J_2(x^{**})$ and $J_0^{**}(x^{**})$ are two distinct Hahn-Banach extensions of $x^{**}$.
Now from a slightly different viewpoint, let $x^{**}\in B$ and $x^*\in X^*(\hookrightarrow X^{***})$ as above; then there exists non-zero $x^{*\perp}\in X^{(4)}$ such that $x^{*\perp}|_{X^*}=0$ and $\|J_2(x^{**})|_{X^*}\|=d(J_2(x^{**}),(X^*)^\perp)=\|J_2(x^{**})+x^{*\perp}\|$. Consequently,
$J_2x^{**}+x^{*\perp}\in X^{(4)}$ is a norm-attaining, and both $J_2x^{**}$ and $J_2x^{**}+x^{*\perp}$ are distinct norm-attaining Hahn-Banach extensions of $x^{**}$. This concludes the following. Here $I$ stands for the identity mapping.

\begin{Cor}\label{C1}
Let $X$ be a Banach space and $B$ be a James boundary for $X^*$, then each of the following is equivalent to reflexivity and hence \tfae.
\bla
\item For every $x^{**}\in B$ there exists a unique Hahn-Banach extension over $X^{***}$. 
\item $I:(B_{X^{**}},w^*)\ra (B_{X^{**}},w)$ is continuous on $ext (B_{X^{**}})\cap NA_1(X^*)$.
\item $I:(B_{X^{**}},w^*)\ra (B_{X^{**}},w)$ is continuous on $B_{X^{**}}$.
\el
\end{Cor}

\begin{ex}\label{E1}
From Corollary~\ref{C1}, the Banach space $\ell_\iy$ can not have property-$(wU)$, when it is considered as a canonical image in its bidual. Now recall that $\ell_\iy^*\cong c_0^\perp\oplus_{\ell_1}\ell_1$ (Yoshida-Hewitt decomposition). Hence $\ell_\iy^{**}\cong c_0^{\perp\perp}\oplus_{\ell_\iy}\ell_1^\perp\cong \ell_\iy\oplus_{\ell_\iy}\ell_1^\perp$. It is clear that the $M$-summand $\ell_\iy$ has property-$(U)$ in $\ell_\iy^{**}$. Although in a dual space an $M$-summand is clearly $w^*$-closed  (see \cite[Theorem~1.9]{HWW}) and hence the aforesaid $M$-summand $\ell_\iy$ can not be the canonical image in  $\ell_\iy^{**}$.
\end{ex}

\begin{Rem}
A similar conclusion can be drawn for $X^{**}$ when $X$ is an $M$-embedded space. More precisely, let $X$ be an $M$-embedded space; then $X^{(4)}$, the fourth dual of $X$, has a $w^*$-closed subspace $W$ where $W\cong X^{**}$ and $W$ has property-$(U)$ in $X^{(4)}$ although $J_2X^{**}$ cannot have property-$(wU)$ in $X^{(4)}$. In fact, from the discussion before Corollary~\ref{C1} there exists a linear functional $\La\in X^{***}$ such that $\La$ has distinct Hahn-Banach extensions $\tilde{\La}_1$ and $\tilde{\La}_2\in X^{(5)}$. 
\end{Rem}

Note that the example in \ref{E1} is an $L_1$-predual. We now show that an $L_1$-predual which is Hahn-Banach smooth, is not necessarily $M$-embedded.

\begin{ex}\label{E2}
Recall the example stated in \cite{O} and also discussed in \cite{JW}. It is a renorming on $c_0$ defined as follows.

For $0<r<1$ define $Y_r=(c_0,\|.\|_r)$. 

Here $\|(\al_n)\|_r=\sup\left\{\frac{|\al_1|}{r}, |\al_1-\al_2|, \ldots, |\al_1-\al_n|, \ldots\right\}$ for $(\al_n)\in c_0$.

In \cite{JW} the authors have derived that the dual norm on $Y_r^*$ is given by $\|(a_n)\|=r|\sum_{n=1}^\iy a_n|+\sum_{n=2}^\iy |a_n|$ for $(a_n)\in Y_r^*$.

In \cite{O}, Oja observed that $Y_r$ has property-$(U)$ in $Y_r^{**}$. In \cite{JW} Johnson and Wolfe observed that if $P:Y_r^{***}\ra Y_r^*$ is the canonical projection, then $\|I-P\|=1+r$ and hence it follows that $Y_r$ cannot be an $M$-ideal in its bidual. 
We now show that the dual norm in $Y_r^*$ is isometrically isomorphic to $\ell_1$ with its usual norm. Define, $$T:Y_r^*\to\ell_1\text{  by  }T\big((a_n)\big)=\big( r\sum_{n=1}^\iy a_n,a_2,a_3,...\big).$$ It is easy to observe that $T$ is an isometry and onto. This concludes that $Y_r$ is an $L_1$-predual. 
\end{ex}

\begin{lemma}\label{L2}
Let $X$ and $Y$ be two Banach spaces such that $X\cong Y$. Then $\|I-P\|=\|I-Q\|$, where $P:X^{***}\to X^*$ and $Q:Y^{***}\to Y^*$ are projections such that $P(x^{***})=x^{***}|_X$ and $Q(y^{***})=y^{***}|_Y$.
\end{lemma}

\begin{Rem}
Aided by Example~\ref{E2} and Lemma~\ref{L2} one can conclude that there exist an uncountable collection of non-isometric preduals of $\ell_1$ which are Hahn-Banach smooth but not $M$-embedded.
\end{Rem}

From Example~\ref{E2} it is clear that $Y_r\ncong c_0(\Ga)$ for any discrete space $\Ga$ (see \cite[Proposition~III.2.7]{HWW}).
In section~4, we answer the Problem~\ref{P8} in details.

\section{Characterization of $L_1$-preduals which are Hahn-Banach smooth}

\subsection{On duals of $M$-embedded spaces}

We refer the reader \cite[Proposition~III.2.10]{HWW} in the context of the following theorem.
\begin{thm}\label{T6}
Let $X$ be an $M$-embedded space and $Y$ be a predual of $X^*$. Also, let $T:Y^*\to X^*$ be the isometric isomorphism. Let $y^*\in S_{Y^*}$. Then \tfae.
\bla
\item $I:(B_{Y^*},w^*)\to (B_{Y^*},w)$ is continuous at $y^*$.
\item $T:Y^*\to X^*$ is relatively $w^*(Y)-w$ continuous at $y^*$.
\item $T:Y^*\to X^*$ is relatively $w^*(Y)-w^*(X)$ continuous at $y^*$.
\el
\end{thm}
\begin{proof}
As $T$ is $w-w$ continuous, hence $(a)\
\Ra(b)$ follows. $(b)\Ra (c)$ is obvious.

$(c)\Ra (a).$ Let $y^*\in S_{Y^*}$ and $(y^*_\al)\subseteq B_{Y^*}$ such that $y^*_\al\stackrel{w^*(Y)}\longrightarrow y^*$. Hence from $(c)$, we get $Ty^*_\al\stackrel{w^*(X)}\longrightarrow Ty^*$. Since $X$ is Hahn-Banach smooth and $Ty^*\in S_{X^*}$, we have $Ty^*_\al\stackrel{w}\longrightarrow Ty^*$. This follows that $y^*_\al\stackrel{w}\longrightarrow y^*$ as $T^{-1}$ is $w-w$ continuous. This completes the proof.
\end{proof}

\begin{thm}\label{T4}
Let $X$ be an $L_1$-predual where $I:(B_{X^*},w^*)\ra (B_{X^*},w)$ is continuous on $ext (B_{X^*})$. Then $(ext(B_{X^*}),w^*)$ is a discrete space.
\end{thm}
\begin{proof}
From the given assumption, $X^*\cong L_1(\mu)$ for some positive measure space $(\Omega,\Si,\mu)$. This follows from \cite[p.144]{JBC} that $ext (B_{X^*})=\{\la \chi_A: A \mbox{~is an atom}, \|\la\chi_A\|=1\}$. Let $A$ be an atom and let $\al>0$ such that $\|\al\chi_A\|=1$. Consider $\La=\al\chi_A\in L_\iy(\mu)$ and the slice $S=\{h\in B_{L_1(\mu)}:\La (h)>1/2\}$. Then for an atom $B\neq A$, $\chi_B\notin S$, and so is $\frac{1}{\mu (B)}\chi_B$. This concludes that $(ext (B_{X^*}),w)$ is a discrete space, and so is $(ext (B_{X^*}),w^*)$, from the given assumption. 
\end{proof}

\subsection{Structure of $L_1(\mu)$ which admits a Hahn-Banach smooth predual}

The following result is well-known in the literature of $L_1$-preduals. Here we add a proof of it for the sake of completeness.

\begin{Pro}\label{P3}
Let $X$ be an $L_1$-predual and $x^*\in ext (B_{X^*})$. Then $\ker (x^*)$ is an $M$-ideal in $X$.
\end{Pro}
\begin{proof}
Note that if $x^*\in ext (B_{X^*})$ then $x^*=\al\chi_A$ for some atom $A$ and $\al\in\mb{R}$ such that $\|\al\chi_A\|=1$. Now from \cite[Example~I.1.6]{HWW} it follows that $P:X^*\ra span\{\chi_A\}$ is an $L$-projection, and hence the result follows.
\end{proof}

\begin{thm}\label{T7}
Let $(\Omega,\Sigma,\mu)$ be a measure space such that $L_1(\mu)$ has a predual which is weakly Hahn-Banach smooth. Then $L_1(\mu)\cong \ell_1(\Gamma)$ for some discrete set $\Gamma$. 
\end{thm}

\begin{proof}
We follow the techniques used in \cite[Theorem~3.11]{BR}.
Let $A= ext(B_{L_1(\mu)})$. Note that $B_{L_1(\mu)}=\overline{conv}^{w^*}\{A\}$. From Theorem~\ref{T4} it follows that $(A,w^*)$ is a discrete space. From Proposition~\ref{P3} it now follows that, each $f\in A$, span$\{f\}$ is an $L$-summand. 

{\sc Claim~:} For any $f_1,...,f_n\in A$, we have $B_F=conv\{\pm f_i:i=1,...,n\}$, where $F=span\{f_1,...,f_n\}$. 

Note that if $h\in B_F, \|h\|=1$ and $h=\la_1f_1+\ldots+\la_nf_n$, for some scalars $\la_i, 1\leq i\leq n$. Then $\|\la_1f_1+\ldots+\la_nf_n\|= 1$ and this leads to $|\la_1|+\ldots+|\la_n|= 1$. Changing $f_i$ by $-f_i$ if necessary, we may assume $\la_i\geq 0$. Thus $h\in conv\{\pm f_i:i=1,...,n\}$.

Thus, $\phi:\ell_1(A)\to L_1(\mu)$ defined by $\phi(\al)=\sum_{f\in A}\al(f).f$ is a linear isometry.   
It remains to prove that $\phi$ is onto. 

First observe that $X$ is Hahn-Banach smooth.
Let $f\in L_1(\mu)$ be such that $\|f\|=1$. 
  
Suppose that $(g_\al)\ci conv\{A\}$ such that $g_\al\stackrel{w^*}\longrightarrow f$, where $g_\al=\sum_{i=1}^{n_\al}\la_i^{n_\al} f_i^{n_\al},\,\{f_i^{n_\al}\}_{i=1}^{n_\al}\subseteq A$ and $\sum_{i=1}^{n_\al}|\la_i^{n_\al}|=1$ for all $\al$. This implies $g_\al\stackrel{w}\longrightarrow f$. Consequently, we have $f\in\overline{conv}^{w}\{A\}=\overline{conv}^{\|.\|}\{A\}$.
  
Now observe that $\phi$ is $w-w$ continuous and $\phi(\ell_1(A))$ is weakly closed since $\phi$ is an isometry. For $\al$, define $\Lambda_\al:A\to\mb{R}$ such that $\Lambda_\al(f_i^{n_\al})=\la_i^{n_\al}$ for $1\leq i\leq n_\al$ and $0$ elsewhere with $\sum_{i=1}^{n_\al}|\la_i^{n_\al}|=1$. It is clear that $\phi(\La_\al)=g_\al$ for all $\al$. Now since $\phi(\La_\al)\stackrel{w}\longrightarrow f$ and $\phi(\ell_1(A))$ is weakly closed hence $f\in \phi(\ell_1(A))$. This follows that $L_1(\mu)\cong \ell_1(A)$, where $A$ is a discrete set.
\end{proof}

\subsection{The characterization}

\begin{thm}\label{T8}
Let $X$ be an $L_1$-predual. Then \tfae.
\bla
\item $X$ is Hahn-Banach smooth.
\item There exists an isometric isomorphism $T:X^*\ra \ell_1(\Ga)$, for some discrete space $\Ga$. Let $T_1=T|_{S_{X^*}}$ be the restriction map. Then $T_1:(S_{X^*},w^*(X))\ra (S_{\ell_1(\Ga)},w^*(c_0(\Ga)))$ induces a homeomorphism.
\item Let $T_1=T|_{S_{X^*}}$ be the above restriction map. Then $T_1:(S_{X^*},w^*(X))\ra (S_{\ell_1(\Ga)},w)$ induces a homeomorphism.
\el
\end{thm}
\begin{proof}
$(a)\Ra (b).$ The first part follows from Theorem~\ref{T7}. Since $c_0(\Ga)^*\cong \ell_1(\Ga)$, from Theorem~\ref{T6} the second part follows.

$(b)\Ra (a).$ Since $c_0(\Ga)$ is $M$-embedded, the result follows from Theorem~\ref{T6}.

$(b)\Lglra (c).$ This follows from Theorem~\ref{T6}.
\end{proof}

\section{Identification of PC for $I:(B_{\ell_1},w^*(c))\ra (B_{\ell_1},w)$}

\subsection{Review to the cases for $c_0$ and $c$}

\begin{notn}
Let $X$ be a predual of $\ell_1$. The tuple $(\ell_1,w^*(X))$ refers to the $w^*$-topology on $\ell_1$ induced by $X$. For a Banach space $Z$, we denote the dual space $Z^*$ with the $w^*$-topology as $(Z^*,w^*)$ when there is no chance of ambiguity.
\end{notn}
We restate the cases for $c_0$ and $c$ in this new notation.

\begin{Cor}\label{null-seq}
\bla
\item The identity mapping $I:(B_{\ell_1},w^*(c_0))\ra (B_{\ell_1},w)$ is continuous on $S_{\ell_1}$.
\item The identity mapping $I:(B_{\ell_1},w^*(c))\ra (B_{\ell_1},w)$ is not continuous on $S_{\ell_1}$.
\el
\end{Cor}

\begin{Rem}\label{R1}
We now illustrate Corollary~\ref{null-seq}$(b)$ by identifying a sequence $(f_m)$ in $B_{\ell_1}$ and $f\in S_{\ell_1}$ where $f_m\stackrel{w^*(c)}\ra f$ but $f_m\stackrel{w}\nrightarrow f$. We follow the notations as in \cite[Example~6.5]{DMR} in this discussion. 
Let $(e_n)_n$ be the standard unit vector basis for $c_0$, and let $e_0 = (1,1,\ldots )$. If $(\al_n )_n \in c$, then $(\al_n)_n - (\lim_n \al_n ) e_0 \in c_0$, which implies that $c = \mb{R}e_0 \oplus c_0$. 

Define a mapping $S:c^*\ra\ell_1$ by
\[
S(f) = \left( f(e_0)-\sum_{n\geq 1} f(e_n),f(e_1),f(e_2),\ldots \right).
\]

Note that $S^{-1}((\be_n)_{n=0}^\iy)((\al_n)_{n=0}^\iy)=\sum_{n=0}^\iy\al_n\be_n$,
for each $(\al_n)_n \in c$ and each $(\be_n )_{n=0}^\infty \in \ell_1$, where $\al_0= \lim_n \al_n$. 

Define $f_m \in S_{c^*}$ by $f_m = S^{-1}(e_m)$ for each $m>1$. Then, we derive
\[
f_m ((\al_n)_{n\geq 0}) = \al_{m-1} \to \al_0 = (S^{-1}(1,0,\ldots ))((\al_n )_{n\geq 0})
\]
as $m \to \infty$ whenever $(\al_n)_{n\geq 0} \in c$, which means that $(f_m)_{m>1}$ converges in the $w^*(c)$ sense to $f = S^{-1}(e_1) \in S_{c^*}$. However, $(Sf_m)_{m>1} = (e_m)_{m>1}$ converges in the $w^*(c_0)$-sense to $0$, which is different from $Sf = e_1$. This concludes that $Sf_m\stackrel{w^*(c_0)}\nrightarrow Sf$ and consequently $Sf_m\stackrel{w}\nrightarrow Sf$. Finally, $f_m\stackrel{w}\nrightarrow f$, as $S$ is $w-w$ continuous.
\end{Rem}

\begin{Rem}\label{R2}
Now consider the sequence of subspaces $\hat{c}_0\ci \hat{c}\ci \ell_\iy$. Here $\hat{c}_0=\{(\al_n)_{n=0}^\iy:\al_0=0=\lim_n\al_n\}$. As discussed in section~1, the Banach space $\hat{c}$ that is canonical image of $c$ in $\ell_\iy$ and $\hat{c}_0$ is assumed to be a subspace of $\hat{c}$. Note that with this identification, an element $(\al_n)\in \hat{c}$ is also an element of $\hat{c}_0$ provided $\al_0=0$ and $\lim_n\al_n=0$. We observe that this $\hat{c}_0$ is an $M$-ideal in $\ell_\iy$. One of the characterizations of $M$-ideals viz. the $3$-ball property may be useful to derive this. We refer the reader to p.18 of \cite{HWW} for this concept.
\end{Rem}

\subsection{The characterization}

\begin{thm}\label{T3}
Let $\Theta=\{(a_n)_{n=1}^\iy\in S_{\ell_1}: a_1=0\}$. Then $\Theta$ is the set of PC of $I:(B_{\ell_1},w^*(c))\ra (B_{\ell_1},w)$.
\end{thm}
\begin{proof}
Consider the sequence of subspaces as stated in Remark~\ref{R2}, $\hat{c}_0\ci \hat{c}\ci \ell_\iy$. As discussed above $\ell_\iy\cong (\hat{c})^{**}$. Note that with the identification aforementioned, in Remark~\ref{R2}, $\hat{c}_0$ is an $M$-ideal in $\ell_\iy$ vis-\'{a}-vis any $(a_n)\in\Theta$, $(a_n)|_{\hat{c}_0}\in (\hat{c}_0)^*$ has a unique extension to $\ell_\iy$. Clearly $\|(a_n)\|_1=\|(a_n)|_{\hat{c}_0}\|_1$ and consequently $(a_n)$ has a unique extension to $(\hat{c})^{**}$. This proves $\Theta\ci \{(a_n)\in S_{\ell_1}:(a_n) \mbox{~is a PC of~}I\}$.

For the converse, we first observe that if $\xi\neq 0$ is real and considering $(\theta_n)$ as a sequence in $\ell_1$, where $\theta_n=(0,0,\ldots,\xi,0,\ldots)$, $\xi$ appears at the $n$-th coordinate, then $\lim_n\theta_n= \theta_1$ with respect to the topology induced by $\hat{c}$. On the other hand, $\lim_n\theta_n=0$ with respect to the topology induced by $\hat{c}_0$.

We precisely use the notations used in Remark~\ref{R1} for this part. Now assume that $\be=(a_n)\notin \Theta$, that means $a_1\neq 0$. 
Note that from Theorem~\ref{T8} it suffices to prove that $S: c^*\to \ell_1$ is not $w^*(c)-w^*(c_0)$ continuous at $S^{-1}(\be)$.

Let $f=S^{-1}\big((a_1,0,0,...)\big)\text{ and } g= S^{-1}\big((0,a_2,a_3,...,a_n,...)\big)$. Note that $\be=Sf+Sg$. 

Also, let $f_m=S^{-1}\big(a_1(e_m)\big)$ for $m\in\mb{N}$. 
Clearly, as stated above, $f_m\stackrel{w^*(c)}\longrightarrow f$ and therefore $f_m+g\stackrel{w^*(c)}\longrightarrow f+g$.

Also, as stated above, $S(f_m)\stackrel{w^*(c_0)}\longrightarrow 0$. Hence $S(f_m+g)\stackrel{w^*(c_0)}\longrightarrow Sg\neq S(f+g)$. Whence $S(f_m+g)\stackrel{w^*(c_0)}\nrightarrow S(f+g)$.

This proves $\Theta\supseteq \{(a_n)\in S_{\ell_1}:(a_n) \mbox{~is a PC of~}I\}$ and this completes the proof.
\end{proof}

\section{Existence of non Hahn-Banach smooth predual}

\begin{thm}\label{T1}
Let $X$ be a nonreflexive $M$-embedded space. Then there exists a predual $Y$ of $X^*$ such that $Y$ is not Hahn-Banach smooth.
\end{thm}
\begin{proof}
We follow the construction as stated in the proof of \cite[Proposition III.2.10]{HWW}.
Since $X$ is an $M$-ideal in $X^{**}$, we have $X^{***}=X^*\oplus_1X^\perp$.
Consider $X^\perp$ as the dual space of $X^{**}/X$ and choose $y^{***}\in X^\perp,\|y^{***}\|=1$. We choose $y^{***}\in NA_1(X^{**}/X)$ where $y^{***}(\Phi)=1$ for some $\Phi\in X^{**}/X$, $\|\Phi\|=1$ and a $w^*$-closed  hyperplane $H=\ker (\Phi)$ in $X^\perp$. It is clear that, $$1\leq \|y^{***}+h^{***}\|\text{ for all }h^{***}\in H.$$

Hence, for any $p^*\in S_{X^*}$, $Z:=H\oplus span\{y^{***}+p^*\}$ is $w^*$-closed in $X^{***}$. We choose $p^*\in NA_1(X)$, in particular.

Let $P$ be the projection onto $X^*$ associated to the decomposition $$X^{***}=X^*\oplus Z.$$
For $x^*+z^{***}\in X^*\oplus Z$, we have 
         \beqa
x^*+z^{***}=x^*+[h^{***}+\la (y^{***}+p^*)]=(x^*+\la p^*)+(h^{***}+\la y^{***})\in X^*\oplus_1 X^\perp.\eeqa
Since $\|\la p^*\|=|\la|\leq \|h^{***}+\la y^{***}\|$ for all $\la\in\mb{R}$, $h^{***}\in H$, we get 
         $$\begin{aligned}
             \|P(x^*+z^{***})\|&=\|x^*\|\\
             &\leq \|x^*+\la p^*\|+\|\la p^*\|\\
             &\leq \|x^*+\la p^*\|+\|h^{***}+\la y^{***}\|\\
             &=\|x^*+z^{***}\|.
\end{aligned}$$
So $P$ is a contractive projection onto $X^*$ with $w^*$-closed kernel $Z$ and is different from $\pi_{X^*}$, since $Z\neq X^\perp.$ Also as we have $Z$ is a $w^*$-closed subspace of $X^{***}$, there exists $Y\subseteq X^{**}$ such that $Y^\perp=Z$. Moreover $Y^*\cong (X^{**})^*/Y^\perp=X^{***}/Z\cong X^*$ and $Y^{***}\cong X^{***}=X^*\oplus Y^\perp\cong Y^*\oplus Y^\perp$. As $Z\neq X^\perp$, we get $Y\ncong X$. In particular, we have $Y={}^\perp Z$, where ${}^\perp Z=\{x^{**}\in X^{**}:z^{***}(x^{**})=0\text{ for all }z^{***}\in Z\}$.

{\sc Claim: } $Y$ is not weakly Hahn-Banach smooth. 

To this end, we note that $-p^*\in S_{X^*}=S_{Y^*}$ but $-p^*$ has two distinct extensions over $Y^{**}$. Clearly $-p^*+0\in Y^{***}$ is an extension. Let $z^{***}=y^{***}+p^*$, then clearly $-p^*+(y^{***}+p^*)$ is an extension of $-p^*$, since $z^{***}\in Y^\perp$. Also $\|-p^*+z^{***}\|=\|\big(p^*+(-p^*)\big)+y^{***}\|=\|y^{***}\|=1$. So $-p^*,-p^*+z^{***}$ are two distinct norm preserving extensions of $-p^*$ over $Y^{**}$. Hence, the result follows.
\end{proof}

\begin{Cor}
Let $X$ be an $L_1$-predual which is Hahn-Banach smooth then $X^*$ has a predual which is not weakly Hahn-Banach smooth.
\end{Cor}
\begin{proof}
Follows from Theorem~\ref{T7} and Theorem~\ref{T1}.
\end{proof}

\section{Remarks on the PC in function spaces}
\subsection{Importance of extreme points in the context of unique extensions}

\begin{Pro}\label{P1}
Let $X$ be a Banach space where no point in $ext (B_{X^*})$ is a PC of $I:(B_{X^*},w^*)\rightarrow (B_{X^*},w)$. Then no $M$-ideal in $X$ can be weakly Hahn-Banach smooth.
\end{Pro}
\begin{proof}
Let $J$ be an $M$-ideal in $X$. The result clearly follows from the following observations.
\bla
\item Let $y^*\in B_{J^*}$ be norm-attaining. Then there exists $y_0\in S_J$ such that $y^*(y_0)=1$. As $H=\{y^*\in B_{J^*}: y^*(y_0)=1\}$ is a face in $B_{J^*}$, there exists an extreme point $y_1^*\in B_{J^*}$ such that $y_1^*(y_0)=1$. 
\item An extreme point of $B_{J^*}$ is an extreme point of $B_{X^*}$, if $J$ is an $M$-ideal in $X$.
\item A PC of $y^*\in S_{J^*}$ is also a PC of $S_{X^*}$ (\cite[Theorem~6.2]{DMR}), if $J$ is an $M$-ideal in $X$.
\el
\end{proof}

In the following we use the techniques used in \cite[Theorem~3.15]{BR}.

\begin{Pro}
Let $L$ be a locally compact Hausdorff space and $X=C_0(L)$. Then the identity mapping $I:(B_{X^*},w^*)\to (B_{X^*},w)$ is continuous on $\mu\in S_{X^*}$ if and only if $\mu=\sum_{n=1}^\iy \al_n\epsilon_n\de_{t_n}$, $(\al_n)\ci (0,1]$, $\epsilon_n=\pm 1$, with $\sum_{n=1}^\iy \al_n=1$ and $t_n\in  L$ are isolated points of $L$. 
\end{Pro}

\begin{proof}
First note that $B_{X^*}=\overline{conv}^{w^*}\{\pm \de_k:k\in L\}$ and hence $\mu\in \overline{conv}^{w^*}\{\pm \de_k:k\in L\}$. Since $I$ is continuous on $\mu$, hence $\mu\in \overline{conv}^{w}\{\pm \de_k:k\in L\}=\overline{conv}^{\|.\|}\{\pm \de_k:k\in L\}$. This implies supp$(\mu)$ can atmost be countable. Let supp$(\mu)=\{t_n:n\in\mb{N}\}$. Since $\|\mu\|=1$, we have $\sum_{n=1}^\iy|\mu(\{t_n\})|=1$. Let $\al_n=\epsilon_n\mu(\{t_n\})$, where $\epsilon_n=sgn (\mu(\{t_n\}))$. Then it is easy to check that $\mu=\sum_{n=1}^\iy \epsilon_n\al_n\de_{t_n}$.

{\sc Claim~:} $t_n$ are isolated points of $L$ for all $n\in\mb{N}$. 

If possible, let $n_0\in\mb{N}$ exist such that $t_{n_0}$ is not an isolated point of $L$. We will see that this contradicts $\al_{n_0}=|\mu(\{t_{n_0}\})|>0$. As $t_{n_0}$ is not an isolated point of $L$, we can find a net $(k_\al)_{\al\in\Ga}\subseteq L$ such as $k_\al\to t_{n_0}$ and $k_\al\neq t_{n_0}$ for any $\al$. This follows that $\de_{k_\al}\stackrel{w^*}\longrightarrow \de_{t_{n_0}}$. Now, define a net $\mu_\al$ in $S_{X^*}$ by 
 $$\mu_{\al}=\sum_{n\neq n_0}\al_n\de_{t_n}+\al_{n_0}\de_{k_\al}.$$

It is clear that $\mu_\al\stackrel{w^*}\longrightarrow \mu$. 

Since $\mu$ is a PC of $I$, $\mu_\al\stackrel{w}\longrightarrow \mu$. This means that $x^{**}(\mu_\al)\to x^{**}(\mu)$ for all $x^{**}\in X^{**}$.

It is well known that $X^*=M(L)$, the collection of all regular Borel measures on $L$, with the total variation norm. Now for $t\in L$, we define $\chi_t: M(L)\to \mb{R}$ by $\chi_t(\mu)=\mu(t)$ for all $\mu\in M(L)$. It is easy to check that  $\chi_t$ is linear. Also, we have $|\chi_t(\mu)|=|\mu(t)|\leq\|\mu\|$, hence $\chi_t\in X^{**}$ for all $t\in L$.

As we have $\mu_\al\stackrel{w}\longrightarrow\mu$, we have $\chi_{t_{n_0}}(\mu_\al)\to \chi_{t_{n_0}}(\mu)$. This implies $\al_{n_0}\de_{k_\al}(t_{n_0})\to \al_{n_0}\de_{t_{n_0}}(t_{n_0})=\al_{n_0}$. But $\de_{k_\al}(t_{n_0})=0$ since $k_\al\neq t_{n_0}$. Hence $\al_{n_0}=0$, a contradiction.

Conversely, assume that $\mu=\sum_{n=1}^\iy \al_n\epsilon_n\de_{t_n}$, $0<\al_n\leq 1$ and $(t_n)_n$ are isolated points of $L$. We have to show that $I$ is continuous on $\mu$. Let $L'$ be the collection of all isolated points of $L$. Clearly, $L'$ is open in $L$, and the subspace topology on $L'$ is the discrete topology.

Now let $D=L\setminus L'$, a closed subset of $L$ and hence $J_D=\{f\in C_0(L): f(t)=0\text{ for all }t\in D\}$ is an $M$-ideal in $C_0(L)$. It is clear that $C_0(L')\cong J_D$. 

Now observe that $C_0(L')$ is Hahn-Banach smooth as $L'$ is discrete and $\mu\in C_0(L')^*$ with $\|\mu\|=1$. Therefore, $\mu$ is a PC of $I:(B_{C_0(L)^*},w^*)\to (B_{C_0(L)^*},w)$.

 Again, $C_0(L')$ is an $M$-ideal in $C_0(L)$, and hence $\mu\in C_0(L)^*$ is a PC of $I:(B_{X^*},w^*)\to (B_{X^*},w)$. Hence, the result follows.
 \end{proof}
\begin{Cor}\label{C4}
Let $L$ be a locally compact Hausdorff space and $X=C_0(L)$. Also, let $I:(B_{X^*},w^*)\to (B_{X^*},w)$ be the identity map. Then
    \bla 
    \item $I$ is continuous at $\de_t\in S_{X^*}$ if and only if $t$ is an isolated point of $L$.
    \item If $I$ is continuous on some $x^*\in S_{X^*}$, then $I$ is continuous at $\de_t$ for some $t\in L$.
    \el 
\end{Cor}


\subsection{Results on spaces of affine functions}

We refer the reader to the Appendix of this article, where a brief review to the convexity theory is given.
Recall that the {\it complementary set} of a face $F$ is the union of all faces that are disjoint from $F$, often denoted by $F'$ (see \cite{EM}).

\bdfn
Let $S$ be a compact convex set, and let $F\ci S$ be a face. $F$ is said to be a {\it split face} if the complementary set $F'$ is convex and for any $s\in S\sm (F\cup F')$, $s$ has a unique representation $s=\la p+(1-\la)q$, where $0<\la<1$, $p\in F, q\in F'$.
\edfn

We need the following lemma to derive our next result.

\begin{lemma}\label{L1}\cite{BPR}
Let $S$ be a compact convex set such that each point of $ext (S)$ is split. If $F$ is a closed split face of $S$ and $F'$ its complementary face and $t_1,t_2,\ldots,t_n\in ext (S)\sm ext (F)$ then there exists $b\in A(S)$ such that $b|_F=0$ and $b(t_j)=1$, $1\leq j\leq n$.
\end{lemma}

Recall that if $S$ is a Choquet simplex, then each face of $S$ is a split face, and hence so is $\{p\}$ if $p\in ext(S)$(see \cite[p.133]{AE}). 

\begin{Pro}\label{E4}
Let $S$ be a compact convex set in a locally convex space $E$ where each $p\in ext (S)$ is a limit point of $ext (S)$ and $\{p\}$ is a split face of $S$. Then no subspace of $A(S)$ can have property-$(wU)$ in $A(S)^{**}$.
\end{Pro}
\begin{proof}
 Without loss of generality, assume that  $Y$ is a subspace of $A(S)$ that has the property-$(wU)$ in $A(S)$. First choose $f\in Y$ and $t\in ext (S)$ so that $\|f\|_{\iy}=|f(t)|=1$. Then clearly $\al\de_t\in Y^*$ and $\al\de_t\in A(S)^*$ is the unique linear extension of itself because $Y$ has property-$(wU)$ in $A(S)$. Here $\al=sgn (f(t))$.

Let $X=A(S)$. Suppose that $Y$ has property-$(wU)$ in $A(S)^{**}$. This implies $\de_t$ is a PC of the identity map $I:(S_{X^*},w^*)\to (S_{X^*},w)$. Let $(t_\al)\subseteq ext (S)$ such that $t_\al\to t$ and $t_\al\neq t$ for all $\al$. Hence, we have  $\de_{t_\al}\stackrel{w^*}\longrightarrow\de_{t}$ and this implies $\de_{t_\al}\stackrel{w}\longrightarrow\de_{t}$. Thus, we have $\de_t\in \overline{conv}^{w}\{\de_{t_\al}:\al\in\Ga\}$ which implies $\de_t\in \overline{conv}^{\|.\|}\{\de_{t_\al}:\al\in\Ga\}$. Therefore, there exists $(\mu_\be)\subseteq conv\{\de_\al:\al\in\Ga\}$ such that $\mu_\be\rightarrow\de_t$ in the norm, and this implies that for $0<\e<1$ there exists $\be$ such that $\|\mu_{\be}-\de_t\|<\e$, say $\mu_{\be}=\sum\limits_{i=1}^k\la_i\de_{t_{\al_i}}$ for some $\la_i\in [0,1],\sum\limits_{i=1}^k\la_i=1$ and $t_{\al_i}\in ext (S)$. But as $t_{\al_i}\neq t$, from Lemma~\ref{L1} one can get $b\in A(S)$ such that $b(t_{\al_i})=1$, $1\leq i\leq k$ and $b(t)=0$. Hence, we have $\|\sum\limits_{i=1}^k\la_i\de_{t_{\al_i}}-\de_t\|\geq 1$, a contradiction.
\end{proof}

Note that for a compact convex set $S$ in a locally convex space $E$, $B_{A(S)^*}=conv \{S\cup -S\}$. Hence $ext (B_{A(S)^*})=\{\pm\de_t:t\in S\}$.

\begin{Cor}\label{C3}
Let $S$ be a Choquet simplex as stated in Proposition~\ref{E4}. Then no point in $ext (B_{A(S)^*})$ is a PC of $I:(B_{A(S)^*},w^*)\ra (B_{A(S)^*},w)$.
\end{Cor}

The following follows as a consequence of Proposition~\ref{P1} and Corollary~\ref{C3}.

\begin{Cor}\label{C2}
Let $S$ be a compact convex set as stated in Proposition~\ref{E4}. Then no $M$-ideal in $A(S)$ is weakly Hahn-Banach smooth.
\end{Cor}

\subsection{Final remarks}

\begin{Cor}\label{accum}
Let $L$ be a locally compact Hausdorff space and $X=C_0(L)$. Then, the following are equivalent.
\bla
\item The identity mapping $I:(B_{X^*},w^*)\ra (B_{X^*},w)$ is continuous on $ext (B_{X^*})$.
\item $L$ is a discrete space.
\el
\end{Cor}
\begin{proof}
It suffices to prove $(a)\Ra (b)$. Note that if $t\in L$ is a limit point of $L$, then from Proposition~\ref{E4}, one can see that for $t\in L, \de_t\in S_{X^*}$ is not a PC of $I$. This concludes $(a)\Ra (b)$.
\end{proof}

\begin{Cor}\label{C5}
Let $L$ be a locally compact Hausdorff space and $X=C_0(L)$. Then \tfae.
\bla
\item $L$ has no isolated point.
\item No point in $ext (B_{X^*})$ can be a PC of $I:(B_{X^*},w^*)\ra (B_{X^*},w)$.
\item No point in $S_{X^*}$ can be a PC of $I:(B_{X^*},w^*)\ra (B_{X^*},w)$.
\item No $M$-ideal of $X$ can be weakly Hahn-Banach smooth.
\item No subspace of $X$ can have property-$(wU)$ in $X^{**}$.
\item No 1-dimensional subspace of $X$ can have property-$(wU)$ in $X^{**}$.
\el
\end{Cor}
\begin{proof}
$(a)\Ra (b).$ Follows from Corollary~\ref{C3}.

$(b)\Ra (c) ~\&~ (c)\Ra (a).$ Follows from Corollary~\ref{C4}.

$(c)\Ra (d).$ Follows from Proposition~\ref{P1}.

$(d)\Ra (a).$ Suppose that $L$ has an isolated point, $\{t\}$. Let $D=L\setminus\{t\}$, then $J_D=\{f\in C_0(L):f(s)=0\text{ for all }s\in D\}$ is an $M$-ideal ( also $M$-summand) in $X$ and $\dim(J_D)=1$. This means that $J_D$ is Hahn-Banach smooth, a contradiction.

$(c)\Ra (e).$ If possible, assume that there exists a subspace of $X$ which has property-$(wU)$ in $X^{**}$. This implies that there exists $x^*\in S_{X^*}$ which is a $w^*-w$-PC, a contradiction to $(b)$.

$(e)\Ra (a).$ If possible, assume that there exists an isolated point $t\in L$. Assume $J_D=\{f\in C_0(L):f(s)=0\text{ for all }s\in D\}$. Note that $ J_D^*=span\{\de_t\}$. Since $J_D$ is an $M$-ideal in $X$, $J_D$ has the property-$(U)$ in $X$. Again, since $t$ is an isolated point of $L$, $\de_t$ has a unique norm preserving extension from $X$ to $X^{**}$ (by Lemma~\ref{L1}). Consequently $J_D$ has the property-$(U)$ in $X^{**}$, a contradiction.

$(f)\Ra (a).$ An identical proof will work as stated in $(e)\Ra (a)$.
\end{proof}

\begin{Rem}
Note that for those locally compact Hausdorff spaces $L$ stated in Corollary~\ref{C5}, if $Y$ is any predual of $C_0(L)^*$ then $Y$ cannot be weakly Hahn-Banach smooth. In fact, if $Y$ is weakly Hahn-Banach smooth, then $C_0(L)^*$ becomes a space with the Radon-Nikod\'{y}m property which forces $C_0(L)$ to be an Asplund space.
\end{Rem}

\section*{Appendix}
For a compact Hausdorff space $K$, by $M(K)$ we denote the set of all regular Borel measures on $K$. It is well-known that $M(K)$ plays the role of a dual Banach space of $C(K)$ with respect to the total variation norm in $M(K)$. We  now derive several results concerning the space $A(S)$, which consists of all real-valued affine continuous functions defined on a compact convex set $S$ from a locally convex space $E$.

We start by adding the fundamentals of convexity theory, which are necessary for our next observations. In this context, we direct the reader's attention to Chapter I of \cite{EM} and Chapter II, section~6 of \cite{AE}. Recall that a compact convex set $S$ in a locally convex space $E$ is said to be a {\it Choquet simplex} if each point $s\in S$ can be represented by a unique maximal measure $\mu$ on $S$. A {\it Bauer simplex} $S$ is a Choquet simplex where $ext (S)$ is closed with respect to the topology endowed on $S$. It is important to note that when $K$ is a compact Hausdorff space, the set of all probability measures on $K$, denoted as $P(K)$, represents a $w^*$-compact convex set within $M(K) (\cong C(K)^*)$. $P(K)$ forms a Bauer simplex, and we have $A(P(K))\cong C(K)$. In fact, in this case, the topology on $K$ can be transferred to the set $ext (P(K))$ through the homeomorphism $\de:K\ra ext (P(K))$. Here, for $k\in K$, by $\de_k (=\de(k))$ we mean the point evaluation functional at $k$. With this identification, one can extend a measure $\mu\in M(K)$ to a measure on $P(K)$. It is evident that if $\mu\in M(K)$ is a probability measure, then upon extending it to a measure on $P(K)$, it yields the representing measure of $\mu$, which is also maximal.

We now give an example of a simplex $S$, other than the Bauer simplexes, where $ext(S)$ is dense in $S$, known as {\it Poulsen simplex}. 
\begin{ex}\cite[p.130]{AE}
 We construct simplexes $S_n$ in Euclidean $n$-space $\mb{R}^n$, embedded in $\ell_2$ in the natural way.

Let $S_1=[0,\frac{1}{2}]$ and let $S_2=conv\Big(S_1\cup \{(0,\frac{1}{2})\}\Big)\subseteq{\mb{R}}^2.$
Next choose points $y_3,y_2,\ldots,y_k\in S_2$ forming a $\frac{1}{2^{2}}$- net for $S_2$, and let $$z_i=y_i+\frac{1}{2^i}e_i,\quad 3\leq i\leq k, \mbox{~and define}$$
 $$ S_i=conv\Big(S_2\cup \{z_3,\ldots,z_i\}\Big)\subseteq \mb{R}^i\mbox{~for~}3\leq i\leq k.$$
In particular, a $\frac{1}{2^{2}}$-net in $S_2$ gives the choice of $k$, and with that one can define $S_k$. Moreover $S_1\subseteq S_2\subseteq...\subseteq S_k$ and $ext(S_1)\subseteq ext(S_2)\subseteq\ldots\subseteq ext(S_k)$. Also note that $z_i$ is an extreme point in $S_i$.

Now choose a $\frac{1}{2^{k}}$-net in $S_k$ and define $S_{k+1},...$ by the same procedure. Now define $S_0=\big(\bigcup\limits_{n=1}^\iy S_n\big)$ and $S=\overline{S_0}$,
then the construction gives $\overline{ext(S)}=S$.   We claim the following facts for $S$.

{\sc Claim~1:~} $\bigcup\limits_{n=1}^\iy ext(S_n)=ext\Big(\bigcup\limits_{n=1}^\iy S_n\Big)$

This follows from the fact that each $S_n$ is a face of $S_m$, when $m>n$. Finally all $S_n$ are faces of $S_0$.

{\sc Claim~2:~} $K=\overline{\{k_n:n\in\mb{N}\}}=\overline{conv}\{k_n:n\in\mb{N}\}$.

Note that if $z\in \overline{conv}\{k_n:n\in\mb{N}\}$ and $\e>0$ there exists $w\in conv\{k_n:n\in\mb{N}\}$ such that $\|z-w\|<\e/2$. Now $w\in S_m$, for a sufficiently large $m$. Now choose sufficiently large $N\in\mb{N}$ that $\frac{1}{2^N}<\e/2$ and construct $S_N$ from its suitable predecessor $S_r, r>m$. Note that there exists $k_j\in ext(S_N)$ where $\|k_j-w\|<\frac{1}{2^N}<\e/2$. This concludes $\|z-k_j\|<\e$ and hence $z\in \overline{\{k_n:n\in\mb{N}\}}$.

{\sc Claim~3:~} Each point of $ext(S)$ is a limit point of $ext(S)$.

Let $s\in ext(S_0)$.
Since $s\in ext(S)\subseteq S_0$, so there exists $N\in \mb{N}$ such that $s\in S_N$. Now choose $\e>0$ and $M\in\mb{N}$ such that $\frac{1}{2^{M}}<\frac{\e}{2}$. Now we can find $k>M$ such that $S_{k+1}=conv\big( S_k\cup \{z_{k+1}\}\big)$, where $z_{k+1}=y_{k+1}+\frac{1}{2^{k+1}}e_{k+1}$ and $y_{k+1}\in S_k$ is an element from finite $\frac{1}{2^{k}}$-net in $S_k$. Moreover we can choose $y_{k+1}$ such that $\|s-y_{k+1}\|<\frac{1}{2^k}$. We already mentioned that $z_{k+1}\in ext(S_K)\subseteq ext(S_0)$. Also $\|s-z_{k+1}\|=\|s-y_{k+1}-\frac{1}{2^{k+1}}e_{k+1}\|\leq\|s-y_{k+1}\|+\frac{1}{2^{k+1}} <\frac{1}{2^k}+\frac{1}{2^{k+1}}<\frac{1}{2^M}+\frac{1}{2^M}=\e$. Hence $s$ is a limit point of $ext(S_0)$.

Finally, if $k\in ext (S)$ but $k\neq k_n$, where $k_n$'s are defined above. Then clearly it is a limit point of $ext(S)$. This establishes claim 3.
\end{ex}

\begin{center}
  \textbf{Acknowledgment}
\end{center}
The author would like to thank Dr. Ryotaro Tanaka and Dr. Tanmoy Paul for suggestions and remarks at the initial stage of this research.

\begin{center}
  \textbf{DECLARATION}
\end{center}
The authors declare that no data is generated during this investigation, and hence the 'data sharing' is not applicable for this research. In addition, the authors declare that there is no conflict of interest.

\Addresses
\end{document}